\documentclass[11pt,a4paper,final]{article}

\newif\ifsiamart
\makeatletter%
\@ifclassloaded{siamart171218}%
  {\siamarttrue}%
  {\siamartfalse}%
\makeatother%

\ifsiamart\else
\usepackage{authblk}
\fi
\usepackage{silence}
\usepackage[T1]{fontenc}
\usepackage[utf8]{inputenc}
\usepackage{csquotes}
\usepackage[UKenglish]{babel}
\usepackage{paralist}
\usepackage[shortlabels]{enumitem}
\usepackage[margin=.6in]{geometry}
\usepackage{xcolor}
\usepackage{graphicx}
\usepackage{bbm}
\usepackage{amsmath}
\usepackage{amssymb}
\usepackage{amsfonts}
\usepackage{stmaryrd}
\usepackage{mathrsfs}
\usepackage{mathtools}
\usepackage{epstopdf}
\usepackage{comment}

\usepackage{xspace}
\usepackage{float}
\usepackage{stackengine}

\ifsiamart\else
\usepackage{amsthm}
\usepackage{hyperref}
\hypersetup{%
    linktocpage=true,
    colorlinks=true,
    citecolor=red,
    linkcolor=blue,
    urlcolor=blue,
}
\usepackage[nameinlink,capitalise]{cleveref}
\newcommand{\email}[1]{\href{mailto:#1}{#1}}
\newenvironment{keywords}{\noindent \textbf{Keywords.}}{}

\fi

\usepackage{setspace}
\usepackage{amssymb,bbm,mathrsfs}
\usepackage{mathtools}
\usepackage{multicol}
\usepackage[normalem]{ulem}
\usepackage{caption}
\usepackage{subcaption}
\usepackage{array}
\usepackage{centernot}
\newcommand{\KL}[2]{\text{KL}(#1\,|\,#2)}

\DeclareMathOperator{\supp}{supp}

\def\W{\mathcal W}

\def\Hc{\mathcal H}

\def\P{\mathcal P}

\def\R{\mathbb R}

\newcommand{\Wass}{\mathscr{W}}

\newcommand{\mb}{\mathbb}
\DeclareMathOperator*{\E}{\mb{E}}

\DeclareMathOperator*{\amin}{argmin}

\DeclareMathOperator{\Id}{I}

\newcommand{\Hess}[1]{\text{Hess}\left( #1 \right)}

\renewcommand{\d}{\mathrm{d}}
\newcommand{\ddt}{\frac{\d}{\d t}}

\newcommand{\li}{\liminf_{n\rightarrow\infty}}
\newcommand{\ls}{\limsup_{n\rightarrow\infty}}
\newcommand{\norm}[1]{\left\lVert#1\right\rVert}

\def\<#1,#2>{\left\langle #1,\,#2\right\rangle}

\ifsiamart
\newtheorem{remark}[theorem]{Remark}
\newtheorem{assumption}{Assumption}
\else
\theoremstyle{plain}

\newtheorem{theorem}{Theorem}
\newtheorem{lemma}[theorem]{Lemma}
\newtheorem{corollary}[theorem]{Corollary}
\newtheorem{proposition}[theorem]{Proposition}

\newtheorem{definition}[theorem]{Definition}

\newtheorem{assumption}{Assumption}

\makeatletter
\newcommand{\mathleft}{\@fleqntrue\@mathmargin0pt}
\newcommand{\mathcenter}{\@fleqnfalse}
\makeatother
\fi

\crefname{lemma}{Lemma}{Lemmas}
\crefname{remark}{Remark}{Remarks}
\crefname{assumption}{Assumption}{Assumptions}
\crefname{proposition}{Proposition}{Propositions}
\crefname{section}{Section}{Sections}
\crefname{subsection}{Subsection}{Subsections}
\crefname{equation}{}{}
\Crefname{equation}{Equation}{Equations}
\newlist{lemmaenum}{enumerate}{3}
\setlist[lemmaenum]{label=(\alph*),ref=\,(\alph*)}
\crefname{lemmaenum}{Lemma}{Lemmas}
\newlist{assumpenum}{enumerate}{5}
\setlist[assumpenum]{label=(\alph*), font={\bfseries}}
\newlist{auxenum}{enumerate}{2}
\setlist[auxenum]{label=(\alph*),ref=(\alph*)}
\crefname{auxenumi}{Item}{Items}
\crefname{enumi}{}{}
\crefname{equation}{}{}
\crefname{assumpenumi}{}{}
\crefname{assumpenumii}{}{}
\Crefname{assumpenumi}{Assumption}{Assumptions}
\Crefname{assumpenumii}{Assumption}{Assumptions}
\Crefname{assumpenumii}{Assumption}{Assumptions}
\crefrangeformat{assumpenumi}{#3#1#4--#5#2#6}
\Crefname{lemmaenumi}{Part}{Parts}
\Crefname{figure}{Figure}{Figures}
\numberwithin{equation}{section}
\numberwithin{theorem}{section}

\usepackage{tikz}
\usetikzlibrary{shapes.misc}
\usetikzlibrary{positioning}

\usepackage[url=true,backend=biber,style=trad-abbrv,url=false,doi=false,isbn=false]{biblatex}
\DeclareFieldFormat{volume}{volume \textbf{#1}}
\DeclareFieldFormat[article]{volume}{\textbf{#1}}

\ifsiamart\else
\let\oldparagraph=\paragraph
\renewcommand\paragraph[1]{\oldparagraph{#1.}}
\fi

\begin{document}
\title{Computing Optimal Transport Plans via Min-Max Gradient Flows}

\ifsiamart
\else
    \author[1]{Lauren Conger$^{a,}$}
    \author[1]{Franca Hoffmann$^{b,}$}
    \author[1]{Ricardo Baptista$^{c,}$}
    \author[1]{Eric Mazumdar$^{d,}$}
    \affil[ ]{\footnotesize $^a$\email{lconger@caltech.edu},
        $^b$\email{franca.hoffmann@caltech.edu}, $^c$\email{rsb@caltech.edu},
    $^d$\email{mazumdar@caltech.edu}}
    \affil[1]{\footnotesize Department of Computing and Mathematical Sciences, Caltech, USA}
    \date{}
\fi

\maketitle

\begin{abstract}
We pose the Kantorovich optimal transport problem as a min-max problem with a Nash equilibrium that can be obtained dynamically via a two-player game, providing a framework for approximating optimal couplings. We prove convergence of the timescale-separated gradient descent dynamics to the optimal transport plan, and implement the gradient descent algorithm with a particle method, where the marginal constraints are enforced weakly using the KL divergence, automatically selecting a dynamical adaptation of the regularizer. The numerical results highlight the different advantages of using the standard Kullback-Leibler (KL) divergence versus the reverse KL divergence with this approach, opening the door for new methodologies.
\end{abstract}



\section{Introduction}

Optimally transporting mass from one distribution $\mu$ to another $\nu$ is an important problem in many applications, ranging from sampling to numerically solving partial differential equations. In this work, we reformulate the Kantorovich optimal transport (OT) problem
\begin{align*}
    \inf_{\rho\in\Gamma(\mu,\nu)} \iint c(x,y)\d \rho(x,y) \,,
\end{align*}
as a min-max problem by rewriting the marginal constraints with KL divergences, the weight of which is maximized,
\begin{align*}
    \inf_{\rho} \sup_{\Lambda\in\R}  \iint c(x,y)\d \rho(x,y) + \Lambda \left(\KL{\rho_1}{\mu} + \KL{\rho_2}{\nu}\right) \,.
\end{align*}
In the above, $\Gamma(\mu,\nu)$ is the set of measures whose marginals are $\mu$ and $\nu$, the function $c$ denotes the optimal transport cost, and $\rho_1$ and $\rho_2$ are the marginals of the joint probability measure $\rho$.
The Nash equilibrium of the min-max problem coincides with the minimizer of the OT problem, and this formulation allows us to develop a gradient flow particle method to dynamically approximate the optimal coupling.

Other works have proposed particle methods for solving OT problems, such as \cite{pandolfi_optimal-transport_2023}, which presents a finite-particle method for mass transport problems in the $\Wass$ metric, applied to particles diffusing on a sphere. Prior to that, \cite{liu_approximating_2021}  proposed a particle method for computing OT plans,  and we build upon their numerical method by using a different, more efficient density estimation method and more advantageous particle initializations. \cite{liu_approximating_2021} provides a $\Gamma$-convergence result 
that we extend 
to achieve a Danskin-type theorem (also known as the envelope theorem) for the best-response of the coupling to the soft constraint, when considering the min-max dynamics under timescale separation. We address an open question posed in that paper by proving that the energy functional is displacement convex, leading to gradient flow convergence guarantees.

Viewing the OT problem through the lens of a min-max flow in infinite dimensions allows us to build upon the results in \cite{conger_coupled_2024}, which provides conditions under which convergence of the min-max dynamics in the space of measures holds. In our setting, the energy functional defining the zero-sum game is no longer strongly convex in both players (with one player being the transport plan, and the other player being the regularization parameter enforcing the marginal constraints), and the estimates we prove here provide new results in this setting.

The most popular and best-developed method for computing OT plans is the Sinkhorn algorithm \cite{cuturi_sinkhorn_2013}. Our particle method is not expected to outperform state-of-the-art implementations of the Sinkhorn algorithm in computation time---due in large part to the amount of recent work on numerically tuning the Sinkhorn algorithm for better performance~\cite{altschuler2017near, li2023importance}. The key contributions of our method in comparison to Sinkhorn lies in the greater expressivity of our approximation to the optimal coupling. The Sinkhorn algorithm is a reweighting of points which approximate a probability density, which means the joint distribution achieved is constrained to have support on the samples from the input marginals. Because we use particles instead, we can construct joint distributions with different support than the samples from the target marginals, which is expected to yield more expressivity especially in high dimensions. Further, the entropy regularization term in the Sinkhorn algorithm encourages a joint density which factorizes with increasing regularization, but ours does not, allowing for the possibility of obtaining a solution closer to the true optimal plan, while still maintaining the computational benefits of regularization.

Recent works have explored the connection between the Sinkhorn algorithm and Wasserstein-2 gradient flows, including \cite{karimi_sinkhorn_2024}, which develops a continuous time Sinkhorn flow, encompassing the setting of  \cite{deb_wasserstein_2023}, which formulates the notion of Sinkhorn as a Wasserstein-2 mirror descent and provides convergence results. \cite{chizat_sharper_2024} gives conditions on the cost function for exponential convergence of the dual sub-optimality gap for population-level Sinkhorn. Our work is complimentary in that we do not have the KL regularization term which encourages factorization, but we do have a KL penalty term on the marginals which, when taken to have infinite weight, enforces the marginal constraints exactly and guarantees convergence of the energy to the true optimal cost.

Other approaches for approximating the OT plan exist, with a vast literature. Alternative parametric methods learn potential functions in the OT problem, which define a parametric approximation to the transport plan~\cite{seguy2018large, korotinwasserstein}. Works which compute OT maps using the Benamou-Brenier formulation of the Wasserstein-2 distance include \cite{papadakis_optimal_2014}, which provides a proximal splitting algorithm, and shows equivalence with the alternating direction method of multipliers. Similarly, \cite{li_computations_2018} solves the discretized Schr\"odinger's bridge problem and \cite{carrillo_primal_2022} provides a method for  discretizing a gradient flow in time and solves the resulting subproblems in a way that could instead be applied to an OT problem with  the Benamou-Brenier formulation. Our particle method and corresponding theoretical results take the approach of relaxing both marginal constraints simultaneously with convergence guarantees as the  constraints are more tightly enforced.

\noindent\textbf{Contributions: }
We pose the optimal transport problem as a min-max problem, allowing us to prove convergence of a timescale-separated  min-max flow to the optimal transport plan, in infinite dimensions. To do so we draw on ideas and make connections to singularly perturbed dynamical systems and two-timescale dynamics~\cite{sastry1999nonlinear}. In particular, we generalize ideas from that literature to infinite dimensions. Along the way, we address an open question in \cite{liu_approximating_2021} by proving displacement convexity of the energy functional and convergence of the resulting gradient flow to the minimizer. 

Our analysis allows us to derive a new particle-based method for approximating optimal transport plans that has \emph{convergence guarantees} and which could ultimately be a powerful tool in the optimal transport toolbox. Numerically, we illustrate how the $L^2$ marginal density error can be improved if the shape of the source and target are taken into account when selecting forward KL (mode-seeking) vs reverse KL (moment-matching) for the energy functional when enforcing the soft constraints on the marginals.

\section{Problem Setup}

\subsection{Notation}
Let $\P_2(\R^d)$ denote the space of probability measures with finite second moments on $\R^d$. We use the notation $\rho(x)\d x$ for measures with or without Lebesgue density. For two measures $\mu\in\P(\R^{d_x})$ and $\nu\in\P(\R^{d_y})$, we denote by $\Gamma(\mu,\nu)$  the set of all joint distributions whose marginals are equal to $\mu$ and $\nu$, i.e.,
\begin{align*}
    \Gamma(\mu,\nu) = \left\{ \gamma\in\P(\R^{d_x} \times \R^{d_y})\, | \, \mu=\pi_{1\#} \gamma\,,\, \nu=\pi_{2\#}\gamma  \right\},
\end{align*}
where $\pi_{i\#}\gamma$ denotes the pushforward of $\gamma$ on the $i^{th}$ marginal.
The transport cost function is $c:
\R^{d_x}\times \R^{d_y}\to[0,\infty)$, and if $c(x,y)=\|x-y\|^2$ is the quadratic cost, we denote the corresponding Wasserstein metric by $\Wass$, i.e.,
\begin{align*}
    \Wass(\mu, \nu)^2 = \inf_{\gamma\in\Gamma(\mu, \nu)} \int \|x-y\|^2 \d \gamma(x,y) \,\,.
\end{align*}
Then $\Gamma^*(\mu, \nu)$ is the set of optimal couplings $\gamma\in\Gamma(\mu, \nu)$ achieving the inf above.
For a sequence of measures $(\rho^{(n)})$, narrow convergence is defined as convergence in duality with all continuous bounded functions and denoted $\rho^{(n)}\overset{*}{\rightharpoonup}\bar \rho$. Convergence in $\Wass$ is equivalent to narrow convergence plus convergence of the second moments of $\rho^{(n)}$.

\subsection{Displacement Convexity}
\begin{definition}[Displacement Convexity]\label{def:displacement_convexity}
A functional $G:\P_2(\R^d)\to\R$ is \emph{$\lambda$-displacement convex} if for all $\rho_0,\rho_1$ that are atomless we have
\begin{align*}
    G(\rho_s) \leq (1-s)G(\rho_0) + s G(\rho_1) + \frac{\lambda}{2} \Wass(\rho_0,\rho_1)^2\,,
\end{align*}
where $\rho_s:=((1-s)x+sy)_\# \gamma$ for $\gamma\in\Gamma^*(\rho_0,\rho_1)$.
\end{definition}

When $\lambda>0$, $\lambda$-displacement convexity is analogous to strong convexity for functions on Euclidean space, with the inequality using the Wasserstein-2 distance function instead of the two-norm. 
A sufficient condition for uniform displacement convexity with constant $\lambda$ is 
\begin{align}
\begin{split}
        \int \<x-x',\nabla_x \delta_\rho G[\rho](x) - \nabla_x \delta_\rho G[\tilde \rho](x') > \d \gamma(x,x')  
    \ge \lambda \Wass(\rho,\tilde\rho)^2
\end{split}
\end{align}
for all $\gamma\in\Gamma^*(\rho_0,\rho_1)$.
For details regarding convexity inequalities, see \cite[Proposition 16.4]{Villani07}. 

\subsection{Optimal Transport Problem}
The Kantorovich formulation of the optimal transport problem, parameterized by a cost function $c:\R^{d_x}\times \R^{d_y} \to \R_{\ge 0}$, is solved by finding a joint distribution  $\rho \in \P_2(\R^{d_x} \times \R^{d_y})$ (i.e., a coupling) which minimizes
\begin{align}\label{eq:OT_problem}
    \inf_{\rho\in\Gamma(\mu,\nu)} \iint c(x,y)\d \rho(x,y)\,.
\end{align}
The minimizer $\rho^*\in\Gamma^*(\mu, \nu)$ is the transport plan which moves mass between $\mu$ and $\nu$ in an optimal fashion. Rather than enforcing the marginal constraints exactly, we can solve this problem approximately by solving the relaxed problem
\begin{align}\label{eq:approx_OT}
    \inf_{\rho\in\P_2}   E(\rho,\Lambda; \mu,\nu) \,,
\end{align}
where $E$ is the energy functional
\begin{align*}
    E(\rho,\Lambda; \mu,\nu) = \iint c(x,y) \, \d \rho(x,y)
    + \Lambda \KL{\rho_1}{\mu} + \Lambda \KL{\rho_2}{\nu}  \,,
\end{align*}
with $\rho_1$ and $\rho_2$ the marginals of $\rho$, defined as
\begin{align*}
    \rho_1(x) &= \int \rho(x,y) \d y \,,\quad
    \rho_2(y) = \int \rho(x,y) \d x\,,
\end{align*}
and $\KL{\cdot}{\cdot}$ is the Kullback-Leibler divergence, defined as $\KL{\mu}{\nu}\coloneqq \int \mu(x) \log(\mu(x)/\nu(x)) \d x$ if $\mu\ll \nu$ and $+\infty$ otherwise. For $\mu, \nu\in \P_2$, the KL divergence is always non-negative, and $\mu=\nu$ a.e. if and only if $\KL{\mu}{\nu}=0$. We also consider the related energy functional 
\begin{align*}
    \tilde E(\rho,\Lambda;\mu,\nu)  = \iint c(x,y) \, \d \rho(x,y) 
    + \Lambda \KL{\mu}{\rho_1} + \Lambda \KL{\nu}{\rho_2} 
\end{align*}
with the arguments in the KL-divergence reversed. We experiment with this choice of energy functional in our numerical results section, showing that it may have some advantages compared to the standard KL divergence. 
The optimizer of the unconstrained problem \eqref{eq:approx_OT} is not the same as the solution to \eqref{eq:OT_problem}. However, the min-max optimization problems
\begin{align}\label{eq:min_max_OT}
    \inf_{\rho\in\P_2} \sup_{\Lambda\in\R}  &E(\rho,\Lambda; \mu,\nu)\,, \\
    \inf_{\rho\in\P_2} \sup_{\Lambda\in\R}  &\tilde E(\rho,\Lambda; \mu,\nu)\,. \label{eq:min_max_OT_nonconvex}
\end{align}
have Nash equilibria given by the solution to \eqref{eq:OT_problem}.
While both \eqref{eq:min_max_OT} and \eqref{eq:min_max_OT_nonconvex} have the same Nash equilibria, the KL divergence terms in \eqref{eq:min_max_OT_nonconvex} are not necessarily convex in $\rho$, even if $\mu$ and $\nu$ are log-concave.
In order to solve \eqref{eq:OT_problem} numerically, we will solve \eqref{eq:min_max_OT} using a min-max gradient flow with respect to the metric $\W_2$ as
\begin{align}\label{eq:min_max_dynamics}
\begin{split}
        \partial_t \rho(t) &= -\frac{1}{\beta}\nabla_{\Wass,\rho} E(\rho(t),\Lambda(t);\mu,\nu) \\
    \ddt \Lambda(t) &= \partial_\Lambda  E(\rho(t),\Lambda(t);\mu,\nu)\,,
\end{split}
\end{align}
where $\nabla_{\Wass,\rho}$ denotes the gradient with respect to $\rho$ in the $\Wass$ metric, and  $0<\beta<1$ is a temperature parameter specifying how fast $\rho$ updates relative to $\Lambda$. When $\beta$ is close to zero, \eqref{eq:min_max_dynamics} is a two-timescale dynamical system.

We will prove convergence in the best-response setting, that is, when $\beta \rightarrow 0$. In this setting, $\rho$ is instantly minimizing $E(\cdot,\Lambda;\mu,\nu)$, and the dynamics are given by
\begin{align}\label{eq:best_response_dynamics}
\begin{split}
        \ddt \Lambda(t) &= \partial_\Lambda  E(b[\Lambda],\Lambda;\mu,\nu) \\
    b[\Lambda] &\coloneqq \amin_{\rho\in\P_2} E(\rho,\Lambda;\mu,\nu) \,.
\end{split}
\end{align}
The minimizer of $E(\cdot,\Lambda;\mu,\nu)$ was proven to exist for any fixed $\Lambda>0$ in \cite[Corollary 1]{liu_approximating_2021}, and therefore $b[\Lambda]$ is well-defined for any $0<\Lambda<+\infty$. 
The other timescale-separation regime, when $\Lambda$ instantly best-responds, is no longer an entropy-regularized problem because as $\Lambda \rightarrow +\infty$, the optimization functional $E$ $\Gamma$-converges to the optimal transport problem \eqref{eq:OT_problem} \cite[Theorem 7]{liu_approximating_2021}. 

\section{Results}\label{sec:results}
We use the following assumptions for the proofs of our key results; relaxation is likely possible, and we later show numerical examples that do not satisfy them.
\begin{assumption}[Admissibility]\label{assump:c_mu_nu_nice_enough}
    The cost functional $c(x,y):\R^{d_x}\times \R^{d_y}\to \R_{\ge 0}$ is lower semicontinuous (lsc),
    and input measures $\mu\in\P_2(\R^{d_x})$ and  $\nu\in\P_2(\R^{d_y})$ are such that $\iint c(x,y)\d\mu(x)\d\nu(y)<\infty$.
\end{assumption}
\begin{assumption}[Convexity]\label{assump:c_mu_nu_convex}
    The cost function $c$ and input measures $\mu\in\P_2(\R^{d_x}),\nu\in\P_2(\R^{d_y})$ satisfy
    \begin{itemize}
        \item $\Hess{c(x,y)}\succeq \tilde \lambda \Id$ for all $x\in\R^{d_x},y\in\R^{d_y}$,
        \item $\nabla^2_x \log \mu(x) \preceq -\lambda \Id$ for all $x\in\R^{d_x}$, and 
        \item $\nabla^2_y \log \nu(y) \preceq -\lambda \Id$ for all $y\in\R^{d_y}$.
    \end{itemize}
\end{assumption}
\begin{assumption}[Existence of unique minimizer]\label{assump:unique_minimizer}
    The functional $E(\cdot,\Lambda;\mu,\nu)$ has a unique minimizer for every $\Lambda>0$.
\end{assumption}
Assumption~\ref{assump:unique_minimizer} is more general than Assumptions~\ref{assump:c_mu_nu_nice_enough}-\ref{assump:c_mu_nu_convex}; for example, when Assumptions~\ref{assump:c_mu_nu_nice_enough}-\ref{assump:c_mu_nu_convex} are satisfied with $c(x,y)=h(x-y)$, then Assumption~\ref{assump:unique_minimizer} follows directly \cite[Corollary 1]{liu_approximating_2021}. 
\subsection{Displacement Convexity of Energy}\label{subsec:convexity}
For a fixed $\Lambda>0$, consider the dynamics
\begin{align}\label{eq:rho-flow}
    \partial_t \rho(t) &= -\nabla_{\Wass,\rho} E(\rho(t),\Lambda;\mu,\nu)\,.  
\end{align}
The gradient flow \eqref{eq:rho-flow} with a fixed $\Lambda$ was proposed in \cite{liu_approximating_2021} as a method for approximately solving the optimal transport problem \eqref{eq:OT_problem}, together with a numerical scheme for \eqref{eq:rho-flow}. The displacement convexity and convergence to a steady state of the flow \eqref{eq:rho-flow} were left as an open question. 
In this section, we provide proofs for both. 
 In fact, the evolution of the marginals $\rho_1$, $\rho_2$ for $\rho$ solving \eqref{eq:rho-flow} is not a closed system, unless $\rho(x,y)$ factors. If $\rho$ does factor, we can think of $\rho_1$, $\rho_2$ as solving an infinite dimensional two-player game (the aligned case in \cite{conger_coupled_2024}) with the same energy $E(\rho_1\rho_2,\Lambda;\mu,\nu)$. 
First we will prove displacement convexity of $E$ with respect to $\rho$, and then convergence of the dynamics \eqref{eq:rho-flow} to a unique steady state. To do so, we will first prove displacement convexity of $\Hc_1(\rho)$ and $\Hc_2(\rho)$, where 
\begin{align*}
    \Hc_i(\rho) \coloneqq \int {\pi_i}_\#\rho(x) \log {\pi_i}_\#\rho(x) \d x
    =  \int \rho_i(x) \log \rho_i(x) \d x\,.
\end{align*}
\begin{lemma}\label{lem:entropy_displ_cvx}
    The energy functionals $\Hc_1 $ and $\Hc_2$ are zero-displacement-convex in the space of absolutely continuous measures $\P^{ac}$.
\end{lemma}
\begin{proof}
In order to show that $\Hc_1(\rho)$ is displacement convex in $\rho$, we will show that 
\begin{align*}
    &\iint\bigg( \<x-x',\nabla_x \delta_\rho \Hc_1[\rho](x,y) - \nabla_x \delta_\rho \Hc_1[\tilde \rho](x',y')>
    \\
    &
    + \<y-y',\nabla_y \delta_\rho \Hc_1[\rho](x,y) - \nabla_y \delta_\rho \Hc_1[\tilde \rho](x',y')>\bigg) \d \gamma(z)\ge 0\,,
\end{align*}
 for all absolutely continuous $\rho,\tilde \rho \in \P_2(\R^{d_x}\times \R^{d_y})$, 
  with $\gamma$ an optimal coupling between $\rho$ and $\tilde \rho$, and $z:=[x,x',y,y']$,
The first variation of $\Hc_1(\rho)$ with respect to $\rho$ is 
    \begin{align*}
        \delta_\rho \Hc_1(\rho)(x,y) &=\log \rho_1(x)\,,
    \end{align*}
and therefore the second inner product in the convexity condition is zero. The convexity condition becomes
    \begin{align*}
        \int \<x-x',\nabla_x (\log(\rho_1(x)) - \log(\tilde \rho_1(x')))> \d \gamma(z) \ge 0\,.    
    \end{align*}
The above holds because $\int\rho_1\log\rho_1$ is zero-displacement convex with respect to $\rho_1$; therefore, $\Hc_1$ is also zero-displacement convex with respect to $\rho$. The result follows similarly for $\Hc_2(\rho)$.
\end{proof}
From now on and for the remainder of Section~\ref{sec:results}, let Assumption~\ref{assump:c_mu_nu_nice_enough} hold. Assumption~\ref{assump:c_mu_nu_convex} is needed for Subsection~\ref{subsec:convexity} and Assumption~\ref{assump:unique_minimizer} is crucial for Sections~\ref{subsec:danskin} and \ref{subsec:long_time_behavior}.
\begin{theorem}[Displacement Convexity]\label{thm:displacement_convexity_E}
 Let Assumptions~\ref{assump:c_mu_nu_nice_enough}-\ref{assump:c_mu_nu_convex} hold. The energy functional $E(\cdot,\Lambda;\mu,\nu)$ is $\tilde \lambda + \Lambda \lambda$ displacement convex. 
\end{theorem}
\begin{proof}
    The term $\iint c(x,y) \d \rho(x,y)$ is $\tilde \lambda$-displacement convex in $\rho$. To check the entropy regularization terms, we will show the convexity of $\KL{\rho_1}{\mu}+\KL{\rho_2}{\nu}$. Since
    \begin{align*}
       \KL{\rho_1}{\mu}+\KL{\rho_2}{\nu} = \Hc_1(\rho)+\Hc_2(\rho)
       -\iint  \log( \mu(x)\nu(y)) \d \rho(x,y) \,,
    \end{align*} and by Lemma~\ref{lem:entropy_displ_cvx}, the first two terms are zero-displacement convex, it remains to check the last integral. By Assumption~\ref{assump:c_mu_nu_convex}, the function $-\log( \mu(x)\nu(y))$ is strongly $\lambda$-convex in $(x,y)$, and therefore the integral is $\lambda$-displacement convex in $\rho$.
\end{proof}
Theorem~\ref{thm:displacement_convexity_E} is the key result that provides convergence guarantees for dynamics \eqref{eq:rho-flow}. 
\begin{theorem}[Convergence]\label{thm:cv}
    Let Assumptions~\ref{assump:c_mu_nu_nice_enough}-\ref{assump:c_mu_nu_convex} hold with $\tilde \lambda + \Lambda \lambda>0$. There exists a unique minimizer $\bar \rho$ of \eqref{eq:approx_OT}, and any solution $\rho(t)$ to \eqref{eq:rho-flow} satisfies
    \begin{align*}
    \Wass(\rho(t),\bar \rho) \le e^{-(\tilde \lambda + \Lambda \lambda)t}  \Wass(\rho(t),\rho(0))\,,
\end{align*}
\end{theorem}
\begin{proof}
    Existence of $\bar \rho$ was shown in \cite[Corollary 1]{liu_approximating_2021} and uniqueness follows from strong convexity of $E$ (Theorem~\ref{thm:displacement_convexity_E}). Further, $E$ is lsc and non-negative since the KL-divergences and $\int c(x,y)\d\rho(x,y)$ are each lsc and non-negative. Differentiating $\Wass(\rho(t),\bar \rho)$ in time via \cite[Thm 3.14]{cavagnari_dissipative_2023} together with $(\tilde \lambda + \Lambda \lambda)$-displacement convexity (Theorem~\ref{thm:displacement_convexity_E}) allows us to conclude the estimate.
\end{proof}

\subsection{Danskin's Theorem}\label{subsec:danskin}
Let Assumptions~\ref{assump:c_mu_nu_nice_enough} and \ref{assump:unique_minimizer} hold for this section. Dropping $\mu,\nu$ from the notation for brevity, denote 
$E_d(\Lambda):= E(b[\Lambda],\Lambda)$.
In this section, we prove that \begin{align*}
    \frac{\d}{\d \Lambda} E_d(\Lambda) = \partial_\Lambda E(\rho,\Lambda)|_{\rho=b[\Lambda]}\,,
\end{align*}
that is, the gradient through the best response is equal to the gradient with the best response plugged in after differentiating, which is classically known for finite-dimensional systems as Danskin's theorem. This results in dynamics of the form
\begin{align*}
    \ddt \Lambda(t) = \KL{b_1[\Lambda(t)]}{\mu} + \KL{b_2[\Lambda(t)]}{\nu}\,,
\end{align*}
which allows us to analyze the long-time behavior of $\Lambda$ in this timescale separated setting. Denoting 
\begin{align}\label{def:V}
   V(\Lambda):=\KL{b_1[\Lambda]}{\mu} + \KL{b_2[\Lambda]}{\nu}\,,
\end{align}
we can write 
\begin{align*}
    E_d(\Lambda) = \iint c(x,y) \d b[\Lambda](x,y) + \Lambda V(\Lambda)\,,
\end{align*}
and $\dot\Lambda=V(\Lambda)$. Danskin's theorem crucially relies on continuity of $V(\cdot)$ and $b(\cdot)$; for $V$ this can be shown directly, and for $b$ we prove this by showing convergence of second moments of $b[\Lambda_n]$ for converging sequences $(\Lambda_n)_n$ and proving $\Gamma$-convergence of $E(\cdot,\Lambda_n)$ w.r.t. $\Wass$.
\begin{lemma}[Marginals of best response]\label{lem:best-response-marginals}
Fix $\Lambda>0$ and denote $Z[\Lambda]:= \iint e^{-c(x,y)/\Lambda}\d\mu(x)\d\nu(y)$. Then the product of the marginals of the best response $b[\Lambda]$ satisfies
\begin{align}\label{eq:EL}
    b_1[\Lambda](x)b_2[\Lambda](y)
    =\frac{1}{Z[\Lambda]}e^{-c(x,y)/\Lambda}\mu(x)\nu(y)=:\sigma[\Lambda](x,y)\,,
\end{align}
where the normalization constant $Z=Z[\Lambda]$ is bounded,
\begin{align}\label{eq:Z}
    0<e^{-c^*/\Lambda} \le Z[\Lambda]\le 1 \quad \forall \,\Lambda>0\,,
\end{align}
with $c^*$ the optimal value of \eqref{eq:OT_problem}.
\end{lemma}
\begin{proof}
  Expression \eqref{eq:EL} directly follows from the Euler-Lagrange condition  for $E(\cdot,\Lambda)$ together with the mass constraint $\iint b_1[\Lambda](x)b_2[\Lambda](y)\d x\d y=1$. 
     From Assumption~\ref{assump:c_mu_nu_nice_enough}, the optimal transport problem \eqref{eq:OT_problem} admits a solution $\rho^*\in\P_2$ with finite cost \cite[Ch 6]{ags}. Let
    \begin{align*}
        c^* :=\inf_{\rho\in\Gamma(\mu,\nu)} \iint c(x,y) \d \rho(x,y)<\infty\,.
    \end{align*}
    Then the optimal value of \eqref{eq:approx_OT} is upper-bounded by $c^*$
    \begin{align*}
        \inf_{\rho\in\P_2} E(\rho,\Lambda) \le c^*\,,
    \end{align*}
    since the OT minimizer $\rho^*$ for \eqref{eq:OT_problem} results in $E(\rho^*,\Lambda)=c^*$.
 Rearranging \eqref{eq:EL}, we have a.e. on $\supp(b[\Lambda])$,
 \begin{align*}
     c(x,y)+\Lambda\left(\log\frac{b_1[\Lambda]}{\mu}+\log\frac{b_2[\Lambda]}{\nu}\right)=-\Lambda\log Z[\Lambda]\,.
 \end{align*}
 Integrating against $b[\Lambda]$ and using the previous estimate,
 \begin{align*}
    -\Lambda\log Z[\Lambda]=E_d(\Lambda)\le c^*\,. 
 \end{align*}
\end{proof}
As a direct consequence of Lemma~\ref{lem:best-response-marginals}, we obtain explicit expressions for the marginals of the best response,
\begin{align}
        b_1[\Lambda](x) &= \frac{1}{Z[\Lambda]}\mu(x)\int e^{-c(x,y)/\Lambda} \d \nu(y)
        \label{eq:best-response-marginal-mu}\\
        b_2[\Lambda](y) &= \frac{1}{Z[\Lambda]}\nu(y) \int e^{-c(x,y)/\Lambda} \d \mu(x)\,.
        \label{eq:best-response-marginal-nu}
    \end{align}  
\begin{lemma}[Convergence of Second Moments]\label{lem:cv_second_moments}
    For any positive sequence $\Lambda_n \rightarrow \bar \Lambda>0$, it holds that $\int \norm{z}^2 \d b[\Lambda_n](z) \rightarrow \int  \norm{z}^2\d b[\bar\Lambda](z)$ with $z=(x,y)$.
\end{lemma}
\begin{proof}
It is sufficient to show convergence of the second moments for the marginals, which is equivalent to $\int \norm{z}^2\d\sigma[\Lambda_n](z)\to \int \norm{z}^2\d\sigma[\bar\Lambda](z)$ with $\sigma$ defined in \eqref{eq:EL}. By Lebesgue's dominated convergence theorem and the bound in~\eqref{eq:Z}, 
we have $Z[\Lambda_n]\to Z[\bar\Lambda]$ and convergence of the second moments follows.
\end{proof}
\begin{proposition}[$\Gamma$-convergence]\label{prop:gamma_cv}
   For any positive sequence $\Lambda_n \rightarrow \bar \Lambda>0$  as $n\rightarrow \infty$, it holds
   $$E(\cdot,\Lambda_n) \overset{\Gamma}{\rightarrow} E(\cdot,\bar \Lambda) \text{ as } n\rightarrow \infty \text{ in } \Wass.$$ 
\end{proposition}
\begin{proof}
 Given $\Lambda_n \rightarrow \bar \Lambda$, we have for any sequence $\rho^{(n)} \to \bar\rho$ in $\Wass$, 
    \begin{align*}
        \li E(\rho^{(n)},\Lambda_n) \geq \li \iint c(x,y)\d\rho^{(n)}(x,y) 
        + \li \Lambda_n(\KL{\rho_1^{(n)}}{\mu} + \KL{\rho_2^{(n)}}{\nu})
    \end{align*}
    (see \cite[Chapter 1, (1.2)]{braides_gamma-convergence_2002}), and lower bounds for each term are sufficient. For the first term, since $c\ge 0$, it holds by Fatou's Lemma for weakly converging measures \cite[Theorem 2.4]{feinberg_fatous_2019},
    \begin{align*}
        \li \iint c(x,y)\d\rho^{(n)}(x,y) \ge
         \iint  c(x,y)\d \bar \rho(x,y) \,.
    \end{align*}
    For the second term, the lower bound follows by lsc of the KL-divergence, and we obtain
    \begin{align*}
    \li E(\rho^{(n)},\Lambda_n) \ge E(\bar\rho,\bar \Lambda)\,.
\end{align*}
    Next, we show for any measure $\bar\rho$ there exists a recovery sequence $(\rho^{(n)})\in\P_2$ with $\Wass(\rho^{(n)},\bar\rho)\to 0$ such that
    \begin{align*}
        E(\bar \rho,\bar \Lambda) \ge \ls E( \rho^{(n)},\Lambda_n)\,.
    \end{align*}
    Indeed, selecting the constant sequence $\rho^{(n)}=\bar \rho$, the integral against $c(x,y)$ is already constant, and since $\Lambda_n\to\bar\Lambda$, the $\limsup$ bound follows immediately.
\end{proof}
\begin{lemma}[Continuity of $b$]\label{lem:continuity_b(rho)}
For any sequence $(\Lambda_n)_n>0$, $\Lambda_n \rightarrow \bar \Lambda>0$  as $n\rightarrow \infty$, we have $b[\Lambda_n]\rightarrow b[\bar \Lambda]$ in $\Wass$ .
\end{lemma}
\begin{proof}
The boundedness of second moments, which follows from convergence of second moments via Lemma~\ref{lem:cv_second_moments}, provides tightness of $(b[\Lambda_n])_n$, and so $(b[\Lambda_n])_n$ is precompact in the narrow topology \cite[Prokorov’s
Theorem 5.1.3]{ags}. Thanks to convergence of second moments (Lemma~\ref{lem:cv_second_moments}), the sequence $(b[\Lambda_n])_n$ is also precompact in $\Wass$. Together with $\Gamma$-convergence of $E(\cdot,\Lambda_n)$ from Proposition~\ref{prop:gamma_cv}, we conclude $b[\Lambda_n]\rightarrow b[\bar \Lambda]$ in $\Wass$ \cite[Theorem 1.21]{braides_gamma-convergence_2002}.
\end{proof}
\begin{lemma}[Continuity of $V$]\label{lem:continuity_V}
    $V:(0,\infty)\to[0,\infty)$ defined in \eqref{def:V} is continuous.
\end{lemma}
\begin{proof}
    Substituting the expressions for the marginals of $b[\Lambda]$ obtained in \eqref{eq:best-response-marginal-mu}-\eqref{eq:best-response-marginal-nu}, we can write
\begin{align}\label{eq:V-in-Z}
    V=&-2\log Z+\frac{1}{Z} \int Z_1(x)\log Z_1(x)\d\mu(x) 
    + \frac{1}{Z} \int Z_2(y)\log Z_2(y)\d\nu(y)\,,
\end{align}
with $Z_1[\Lambda](x), Z_2[\Lambda](y)\in (0,1)$ given by
\begin{align*}
Z_1(x):= \int e^{-c(x,y)/\Lambda}\d\nu(y),\, Z_2(y):= \int e^{-c(x,y)/\Lambda}\d\mu(x) \,.
\end{align*}
For any positive sequence $\Lambda_n\to\bar\Lambda$, we have by Lebesgue's dominated convergence theorem that $Z[\Lambda_n]$, $Z_1[\Lambda_n](x)$ and $Z_2[\Lambda_n](y)$ converge (pointwise). By Jensen's inequality, $|Z_1[\Lambda](x)\log Z_1[\Lambda](x)|\le \frac{1}{\Lambda} \int c(x,y)\d\nu(y)$  for all $x\in \supp\mu$, and thanks to Assumption~\ref{assump:c_mu_nu_nice_enough} this upper bound is integrable w.r.t. $\d\mu$. Again by Lebesgue's dominated convergence theorem, 
\begin{align*}
    &\lim_{n\to\infty}\int Z_1[\Lambda_n](x)\log Z_1[\Lambda_n](x)\d\mu(x)
    = \int Z_1[\bar\Lambda](x)\log Z_1[\bar\Lambda](x)\d\mu(x)\,,
\end{align*}
and similarly for the term involving $Z_2$. We conclude that $V(\Lambda_n)\to V(\bar\Lambda)$.
\end{proof}
\begin{theorem}[Danskin]\label{thm:danskin}
$   \frac{\d}{\d \Lambda} E_d(\Lambda) = \partial_\Lambda E(\rho,\Lambda)|_{\rho=b[\Lambda]}$. As a result, the dynamics \eqref{eq:best_response_dynamics} are given by $\dot\Lambda=V(\Lambda)$.
\end{theorem}
\begin{proof}
 Let $\bar\Lambda>0$ and define $\Lambda_h:=\bar \Lambda + h$ with $h\in\R$ small enough so that $\Lambda_h>0$. Let $\rho^{(h)}:=b(\Lambda_h)$.
  By continuity of $b(\cdot)$ w.r.t. $\Wass$ (Lemma~\ref{lem:continuity_b(rho)}), we have that $\rho^{(h)}$ and its marginals converge in $\Wass$ to $b[\bar \Lambda]$ and its marginals, respectively. We can then apply exactly the same argument as in \cite[Proposition B.10]{conger_coupled_2024} to conclude the proof, with the following adjustments: (i) all inequalities and uses of $\liminf$ and $\limsup$ are flipped as $b[\Lambda]$ is a minimizer not a maximizer, with the same result at the end, (ii) our functional $E(\rho,\Lambda)$ is linear in $\Lambda$, (iii) its derivative $\partial_\Lambda E(\rho,\Lambda)$ is only lsc in $\rho$ rather than continuous. Therefore, only the last of the four inequalities in \cite[Proposition B.10]{conger_coupled_2024} needs to be addressed in our setting: for $h<0$, 
  \begin{align*}
      \limsup_{h \uparrow 0} \frac{E_d(\Lambda_h)-E_d(\bar\Lambda)}{h} 
      &\le \limsup_{h \uparrow 0} \partial_\Lambda E(\rho,\bar\Lambda)|_{\rho=b(\Lambda_h)}
      =\limsup_{h \uparrow 0} V(\Lambda_h) =V(\bar\Lambda)\,,
  \end{align*}
  by continuity of $V$ (Lemma~\ref{lem:continuity_V}).
\end{proof}
\subsection{Long-Time Behavior}\label{subsec:long_time_behavior}
In regularized optimal transport problems, it is important for numerical stability that the weight of the regularizer does not grow to infinity too quickly. Under Assumptions~\ref{assump:c_mu_nu_nice_enough} and \ref{assump:unique_minimizer}, we prove that $\Lambda$ grows at most $\mathcal O(\sqrt{t})$, and that the KL divergence terms in $V$ go to zero as $t\rightarrow \infty$, implying convergence of the best response to the optimal transport plan.
By non-negativity of the KL-divergence (thanks to Gibbs' inequality), we have $V\ge 0$, and so $\Lambda(t)$ solving \eqref{eq:best_response_dynamics} is non-decreasing since $\dot\Lambda=V(\Lambda)$ by Theorem~\ref{thm:danskin}.
\begin{theorem}[Regularization growth condition]\label{thm:Lambda-upper}
    Under the dynamics \eqref{eq:best_response_dynamics}, the regularization parameter $\Lambda$ satisfies
    \begin{align*}
        &\Lambda(t) \le \sqrt{2(c^* t + c_0)} \quad \forall t\ge 0\,,
    \end{align*}
    where $c^*<\infty$ is the optimal value of \eqref{eq:OT_problem} and $c_0>0$ is a constant dependent on $\Lambda(0)>0$.
\end{theorem}
\begin{proof}
From $E_d(\Lambda)\le c^*$ (proof of Lemma~\ref{lem:best-response-marginals}), we have $V(\Lambda)\le c^*/\Lambda$,
    and so $\dot \Lambda\le c^*/\Lambda$.
    By Gr\"onwall's inequality,
    \begin{align*}
        \Lambda(t) \le \sqrt{2(c^* t + c_0)}\,.
    \end{align*}
\end{proof}
To characterize the long-time behavior of $E_d$ and $V$, we re-write $V$ and its derivative in compact form using the measure $\sigma$ defined in \eqref{eq:EL}, and show Lipschitzness of $V$.
\begin{lemma}\label{lem:dV_dLambda} 
We have $V=-2\log Z +{\E}_\sigma[\log (Z_1 Z_2)]$,
 and 
 \begin{align*}
    \frac{\d V}{\d\Lambda}&=\frac{1}{\Lambda^2}\mathrm{Cov}_\sigma(c,\log(Z_1 Z_2))=\frac{1}{\Lambda^2} \left( {\E}_\sigma\left[c\log (Z_1 Z_2)\right]-{\E}_\sigma[c] {\E}_\sigma[\log (Z_1 Z_2)] \right)  \,.
\end{align*}
\end{lemma}
\begin{proof}
Using $\sigma$ defined in \eqref{eq:EL} and expression \eqref{eq:V-in-Z} for $V$, we can compute its derivative in $\Lambda$ explicitly, using $\frac{\d Z}{\d\Lambda}=\frac{1}{\Lambda^2} \iint c e^{-c/\Lambda}\d\mu\d\nu$ and
\begin{align*}
    \frac{\d Z_1(x)}{\d\Lambda}=\frac{1}{\Lambda^2} \int c e^{-c/\Lambda}\d\nu, \,
    \frac{\d Z_2(y)}{\d\Lambda}=\frac{1}{\Lambda^2} \int c e^{-c/\Lambda}\d\mu. 
\end{align*}
\end{proof}

\begin{lemma}[Lipschitzness of $V$]\label{lem:V-Lip}
There exists a constant $0<c_0<\infty$ such that 
$$
\left\|\frac{\d V}{\d\Lambda}\right\|_\infty=\sup_{\Lambda\in[0,\infty)}\left|\frac{\d V}{\d\Lambda}(\Lambda)\right|<c_0\,.
$$
\end{lemma}
\begin{proof}
Consider a sequence $\Lambda_n\to\infty$. As in the proof of Lemma~\ref{lem:continuity_V}, we have $Z[\Lambda_n], Z_1[\Lambda_n](x), Z_2[\Lambda_n](y)\to 1$  (pointwise) as $n\to\infty$. Thanks to Lebesgue's dominated convergence theorem again, we can pull the limit $n\to \infty$ into all integrals appearing in the expression of $\d V/\d\Lambda$ from Lemma~\ref{lem:dV_dLambda}. To do so, we use that $|\log Z_1[\Lambda](x)|\le \int c\d\nu/\Lambda$ (and similarly for $Z_2)$ thanks to Jensen's inequality, and $ce^{-c/\Lambda}/\Lambda\le 1/e$. Hence
$$
\left|\frac{\d V}{\d\Lambda}(\Lambda_n)\right|\to 0 \quad \text{ as } n\to \infty\,.
$$
Using a similar argument via Lebesgue's dominated convergence theorem, one can show that for any sequence $\tilde\Lambda_n\to\tilde\Lambda$ it holds $\d V(\tilde\Lambda_n)/\d\Lambda\to \d V(\tilde\Lambda)/\d\Lambda$, and so $\d V/\d\Lambda$ is continuous.
Finally, for our sequence $\Lambda_n\to\infty$, there exists $N>0$ such that $|\d V(\Lambda_n)/\d\Lambda|<1$ for all $n\ge N$, and by continuity $\d V/\d\Lambda$ achieves its maximum (denoted by $c_1$) on $[0,\Lambda_N]$. We conclude by taking $c_0:=\max\{1,c_1\}$.
\end{proof}
\begin{proposition}[Properties $\Lambda(t)$]\label{prop:Lambda-properties}
    The ODE $\dot\Lambda=V(\Lambda)$ with $\Lambda(0)>0$ has a unique solution $\Lambda(\cdot)\in C([0,\infty))$ that satisfies $V(\Lambda(t))\to 0$ as $t\to\infty$. In other words, the best response $b[\Lambda(t)]$ satisfies the marginal constraints asymptotically.
\end{proposition}
\begin{proof}
  For $\Lambda(0)>0$, define $C_0:= \frac{c^*}{\Lambda_0(\Lambda_0+1)}$. Since $\dot\Lambda=V(\Lambda)\ge 0$, we have 
$$
\Lambda(t)\ge \Lambda(0)= -\frac{1}{2}+\sqrt{\frac{1}{4}+\frac{c^*}{C_0}}\,\Rightarrow\,
V(\Lambda)\le \frac{c^*}{\Lambda}\le C_0(1+\Lambda)\,.
$$
for all $t\ge 0$. Thanks to continuity of $V$ (Lemma~\ref{lem:continuity_V}), Lipschitzness of $V$ (Lemma~\ref{lem:V-Lip}) and the growth condition from Theorem~\ref{thm:Lambda-upper}, we obtain global existence and uniqueness of a continuous solution $\Lambda(t)$ to the ODE $\dot\Lambda=V(\Lambda)$.
Further, $\Lambda(t)$ is non-decreasing, so it must either remain bounded for all times, or diverge to $+\infty$. Thus, we can write $\Lambda(t)\to\Lambda_\infty\in [\Lambda(0),\infty]$, and by continuity, $$V(\Lambda(t))\to V_\infty:=V(\Lambda_\infty)\in [0,\infty]\,.$$ 
Note from Lemma~\ref{lem:best-response-marginals}  for all $t\ge 0$,
$$
V(\Lambda(t))= \frac{1}{\Lambda(t)}\left(E_d(\Lambda(t))-\int c\d b[\Lambda(t)]\right)\le \frac{c^*}{\Lambda(0)}\,,
$$
and so $V$ is uniformly bounded from above. Thus, 
there are two possibilities: either (a) $V_\infty\in (0,\infty)$, or (b) $V_\infty=0$.
In Case (a), there exists $T>0$ such that for all $t\ge T$,
\begin{align*}
    \dot\Lambda(t)=V(\Lambda(t))> \frac{V_\infty}{2}>0\,.
\end{align*}
Integrating the ODE, we obtain the lower bound
\begin{align*}
    \Lambda(t)>\frac{\Lambda_\infty}{2} t + V_0 \quad \forall t\ge T\,,
\end{align*}
where $V_0=V_0(T)\in\R$. For large enough $t\ge T$, this is a contradiction to the upper bound in Theorem~\ref{thm:Lambda-upper}, so (a) cannot happen. We conclude $V(\Lambda(t))\to 0$.  
\end{proof}

\begin{corollary}[Convergence of OT Cost and Minimizer]
 The best response $b[\Lambda(t)]\rightarrow\rho^*$ in $\Wass$, where $\rho^*$ is an optimizer for  \eqref{eq:OT_problem}, and $E_d(\Lambda)\rightarrow c^*$ as $t\rightarrow \infty$.
\end{corollary}
\begin{proof}
    Applying Theorem~\ref{thm:danskin}, the time derivative of $E_d$ along solutions to \eqref{eq:best_response_dynamics} is
\begin{align*}
    \ddt E_d(\Lambda) = V(\Lambda)^2\,.
\end{align*}
This indicates that unless $V=0$, the energy will increase. 
From Proposition~\ref{prop:Lambda-properties}, we have that the solution $\Lambda(t)$ to \eqref{eq:best_response_dynamics} converges to $\Lambda_\infty\in [\Lambda(0),\infty]$ as $t\to\infty$.  By continuity of the best response (Lemma~\ref{lem:continuity_b(rho)}), $b[\Lambda(t)]\to b_\infty:=b(\Lambda_\infty)$ in $\Wass$. Note that $V_\infty:=V(\Lambda_\infty)=0$ (Proposition~\ref{prop:Lambda-properties}) means $b_\infty\in\Gamma(\mu, \nu)$
and so $E_d(\Lambda_\infty)=\int c b_\infty \ge c^*$ since $b_\infty$ is admissible for the OT problem \eqref{eq:OT_problem}. Since also $E_d(\Lambda_\infty)\le c^*$ (Lemma~\ref{lem:best-response-marginals}), we have equality, and $b_\infty$ is an optimal plan in $\Gamma(\mu, \nu)$. Finally, $E_d$ is continuous thanks to the continuity of $b$ and $V$, and so $E_d(\Lambda)\rightarrow E_d(\Lambda_\infty)=c^*$. 
\end{proof}

\section{Numeric Implementation}
In practice, Theorem~\ref{thm:danskin} allows us to chose $\dot\Lambda=V(\Lambda)$, rather than taking a gradient of $b[\Lambda]$ with respect to $\Lambda$. In our numerical implementation\footnote{Code available  \href{https://github.com/LEConger/OT-MinMax-GradientFlow}{https://github.com/LEConger/OT-MinMax-GradientFlow}.}, we chose very small  $\beta$, which results in $\rho$ updating much faster than $\Lambda$, and we expect this to approximate the dynamics \eqref{eq:best_response_dynamics} due to the timescale separation results in \cite{borkar_multiple_2008}. We adapt a method from \cite{liu_approximating_2021} by approximating the distribution $\rho$ by particles $(X_i,Y_i)\sim \rho$.
We divide the points $(X_j,Y_j)_{j=1}^{2N}$ into two sets: for one set, denoted $(X^{(1)}_i,Y^{(1)}_i)_{i=1}^N$, we initialize $X_i^{(1)}\sim\mu$ and $Y_i^{(1)}=X_i^{(1)} + \xi_i$, where $\xi_i\sim \mathcal N(0,\eta)$ with small $\eta$, and we evolve $Y_i^{(1)}$ while keeping $X_i^{(1)}$ fixed. For the other set, denoted $(X^{(2)}_i,Y^{(2)}_i)$ we initialize $Y_i^{(2)}\sim\nu$ and set $X_i^{(2)}=Y_i^{(2)} + \xi_i$, and evolve $X_i^{(2)}$ while keeping $Y_i^{(2)}$ fixed, resulting in 
\begin{align*}
    \dot X_i^{(2)}(t) &= -\frac{1}{\beta}\nabla_x c(X_i^{(2)}(t),Y_i^{(2)}) + \frac{1}{\beta}\Lambda(t)\nabla_x \log \frac{\rho_1(t,X_i^{(2)}(t))}{\mu(X_i^{(2)}(t))} \\
    \dot Y_i^{(1)}(t) &= -\frac{1}{\beta}\nabla_y c(X_i^{(1)},Y_i^{(1)}(t)) + \frac{1}{\beta}\Lambda(t)\nabla_y \log \frac{\rho_2(t,Y_i^{(1)}(t))}{\nu(Y_i^{(1)}(t))}  \\
    \dot \Lambda(t) &=\KL{\rho_1(t)}{\mu} + \KL{\rho_2(t)}{\nu}\,.
\end{align*}

 The densities $\rho_1,\rho_2$ are approximated at each time step by the binning estimator
\begin{align*}
    \rho_1(x) &\approx \frac{1}{2N}\sum_{k=1}^{N_b} \mathbf 1_{\Omega_k}\{x \} \sum_{m=1}^{N}  (\mathbf 1_{\Omega_k}\{X_m^{(1)}\}+\mathbf 1_{\Omega_k}\{X_m^{(2)}\}) \\
    \rho_2(y) &\approx \frac{1}{2N}\sum_{k=1}^{N_b} \mathbf 1_{\Omega_k}\{y \} \sum_{m=1}^{N}  (\mathbf 1_{\Omega_k}\{Y_m^{(1)}\}+\mathbf 1_{\Omega_k}\{Y_m^{(2)}\}) \,,
\end{align*}
where the domain is divided into $N_b$ regions, and the fraction of particles in each region $\Omega_k$ is counted. 
The key differences from \cite{liu_approximating_2021} are (i) how we estimate the marginal densities $\rho_1(t),\rho_2(t)$ at each time step, and  (ii) the choice of particle initialization and updates. For (i), our density estimator is significantly faster to compute because the complexity scales as $\mathcal O(N_b)$ instead of $\mathcal O(N^2)$, and we can select $N_b \ll N^2$. For (ii), the update strategy ensures that each marginal has at least half of the particles distributed according to the input distribution, and the correlation between $X$ and $Y$ particles in the initial condition aims to minimize cost functions of the form $c(x,y)=h(x-y)$. We  also compute marginal estimates of $\mu$ and $\nu$ from samples rather than assuming closed-form knowledge, and including a time-varying $\Lambda$.

\paragraph{Method Comparison}
We compare our method with the method proposed in \cite{liu_approximating_2021}, for computing an approximate transport plan between $\mathcal N(0.4\Id_2,\Sigma)$ and $\mathcal N(0.6\Id_2,\Sigma)$ for $\Sigma=0.02\Id_2$. To replicate the method in \cite{liu_approximating_2021}, we used $1000$ points with $dt=0.00002$ for $T=50$ iterations and selected the $\tau$ parameter for best results. This results in a run time of 2 minutes and 14 seconds. In comparison, we implemented our method with $2N=20,000$ points, $dt=0.0005$, and $T=2000$ iterations, which had a run time of only six seconds. In both methods we include additive Gaussian noise in the particle update with standard deviation $0.02\sqrt{dt}$ for smoothing; the results are shown in Figure~\ref{fig:Gaussian_comparison}. In Figures \ref{fig:Gaussian_comparison} and \ref{fig:KL_switching}, we plot instances of the interpolant $\mu_s:=((1-s)x+sy)_\#\rho_\infty$ for different values of $s\in (0,1)$, where $\rho_\infty$ is the coupling at the final time obtained dynamically via the chosen numerical scheme.
\begin{figure}
    \centering
    \includegraphics[width=0.7\linewidth]{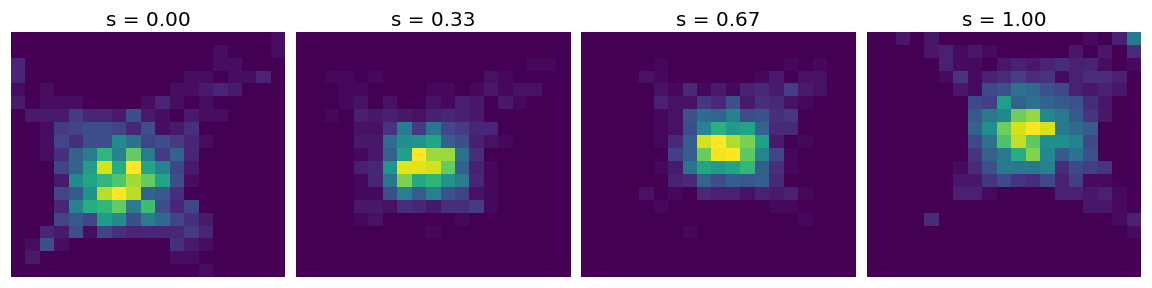} 
    \includegraphics[width=0.7\linewidth, trim={0 0 0 1cm},clip]{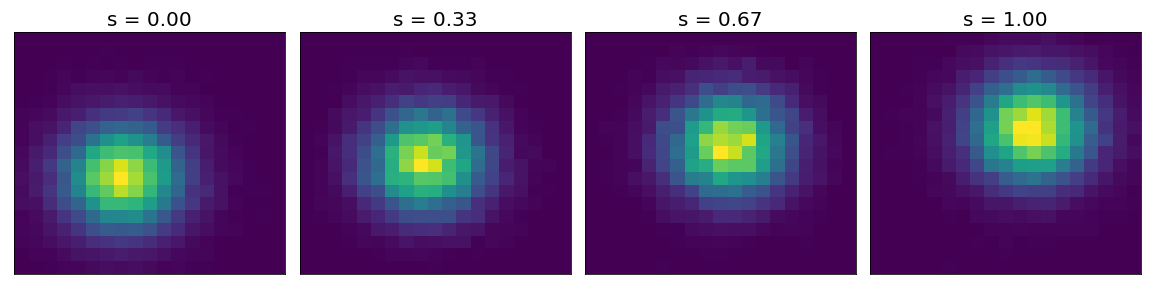}
    \caption{The transport map generated with a kernel density estimator method from \cite{liu_approximating_2021} (top) is limited by the number of points ($N=1,000$), whereas estimating the density by counting voxels (bottom, our method) allows for a greater number of points ($20,000$) and faster computation time.}
    \label{fig:Gaussian_comparison}
\end{figure}

The optimal transport cost $c^*$ for $c(x,y)=\norm{x-y}^2$ can be computed in closed form to be $0.080$; numerically, our method achieves a cost of $0.0861$ (Figure~\ref{fig:coupling_cost_KL_divergence}), with an $L^2$ error between the marginals and input marginals of $0.006$ and KL divergence of $0.098$. This shows the obtained transport plan is very close to the optimal plan.
\begin{figure}
    \centering
    \includegraphics[width=0.55\linewidth]{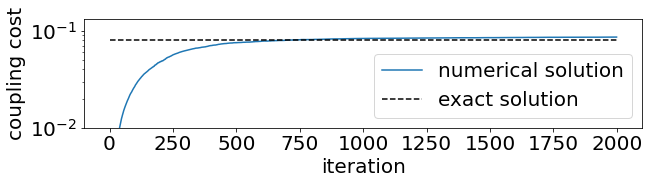} 
    \includegraphics[width=0.55\linewidth]{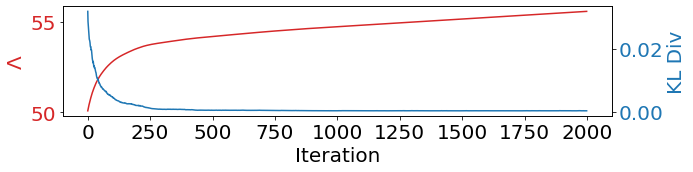}
    \caption{In the setting of two log-concave input measures, the coupling cost starts small due to the initialization of $X_i\approx Y_i$, and increases to approach the optimal coupling cost. The regularization penalty weight $\Lambda$ increases more slowly over time, as expected due to the $\mathcal O (\sqrt{t})$ estimate, and the KL divergence between the marginals and the inputs decreases as the particles drift to minimize the entropy regularization terms.}
    \label{fig:coupling_cost_KL_divergence}
\end{figure}
The computational speedup in comparison with \cite{liu_approximating_2021} is due to the estimation of the kernel density. 
\paragraph{Reversing KL Divergence, Non-Log-Concave Source and Target}
We experiment with non-log-concave reference measures, shown in Figure~\ref{fig:target_distributions}. We also test using the energy $\tilde E$, where the KL terms are reversed. In this gradient flow setting, $\KL{\mu}{\cdot}$ is not necessarily geodesically convex, even if the source marginal $\mu$ is log-concave. Interestingly, we observe better convergence using the reversed KL for a number of examples which coincides with the phenomenon that $\KL{\cdot}{\mu}$ exhibits "mode-seeking" behavior while $\KL{\mu}{\cdot}$ exhibits moment matching behavior~\cite{sanz-alonso_inverse_2023}. In the particular example shown (Figure~\ref{fig:target_distributions}) where the density of $\mu$ has a ring-shape with a central peak and $\nu$ is a mixture of four Gaussian distributions, using $\KL{\mu}{\rho_1}$ to flow points from $\nu$ to $\mu$ results in a smaller $L^2$ error norm, $\int |\rho_1(x)-\rho_1^*(x)|^2\d x$, compared with $\KL{\rho_1}{\mu}$. We suspect this is due to the moment-matching behavior of $\KL{\mu}{\rho_1}$, since the estimated marginal achieved with $\KL{\mu}{\rho_1}$ preserves mass on the outer ring, which is further away from the initial conditions of the point $(X_i^{(2)})$ (as compared to being mode-seeking, which would focus the points more on the mode near the origin).  For the points $(Y_i^{(1)})$ flowing from $\mu$ to $\nu$, the mode-seeking behavior (corresponding to $\KL{\rho_2}{\nu}$) is more advantageous because $\nu$ is very multi-modal. In contrast, when using $\KL{\nu}{\rho_2}$, we observe more outlier points which makes the resulting marginal further from the target distribution.

In our experiments, we found that using randomly selected left or right differences to estimate the gradients outperformed the centered first-order and centered second-order gradients. Using the left and right differences, we compare three settings:
\begin{itemize}
    \item[(I)] Both the $X^{(1)}$ and $Y^{(2)}$ particles are updated via $E$.
    \item[(II)] Both the $X^{(1)}$ and $Y^{(2)}$ particles are updated via $\tilde E$.
    \item[(III)] The $X^{(1)}$ particles are updated via $\tilde E$ and the $Y^{(2)}$ particles are updated via $E$.
\end{itemize}
The simulation results, shown in Figure~\ref{fig:KL_switching}, show that evolution according to (I) and (III) give similar results, while (II) is slightly worse with some rogue particles in corners at the target $\nu$. For all three setting, a large error comes from the source $\mu$. The $L^2$ norm error and KL divergences between the obtained marginals and the true input marginals are listed in Table~\ref{tab:KL_switching}, showing that method (III) has the smallest $L^2$ error and ties for the smallest KL divergence error.  Investigating when $\tilde E$ versus $E$ gives better performance is a future direction of research, along with the displacement convexity properties of $\tilde E$. 

\begin{figure}
    \centering
    \includegraphics[width=0.4\linewidth]{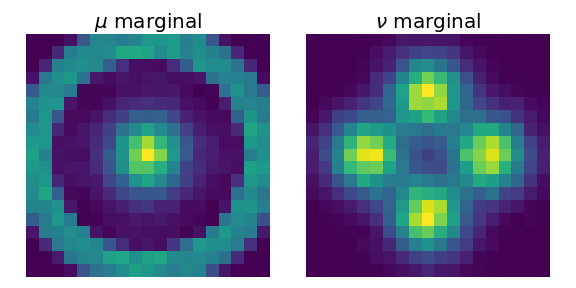}
    \caption{The source and target distributions for the experiment (b) are shown. Note that the input marginals are not log-concave, going beyond Assumption~\ref{assump:c_mu_nu_convex}.}
    \label{fig:target_distributions}
\end{figure}
\begin{figure}
    \centering
    \includegraphics[width=0.65\linewidth]{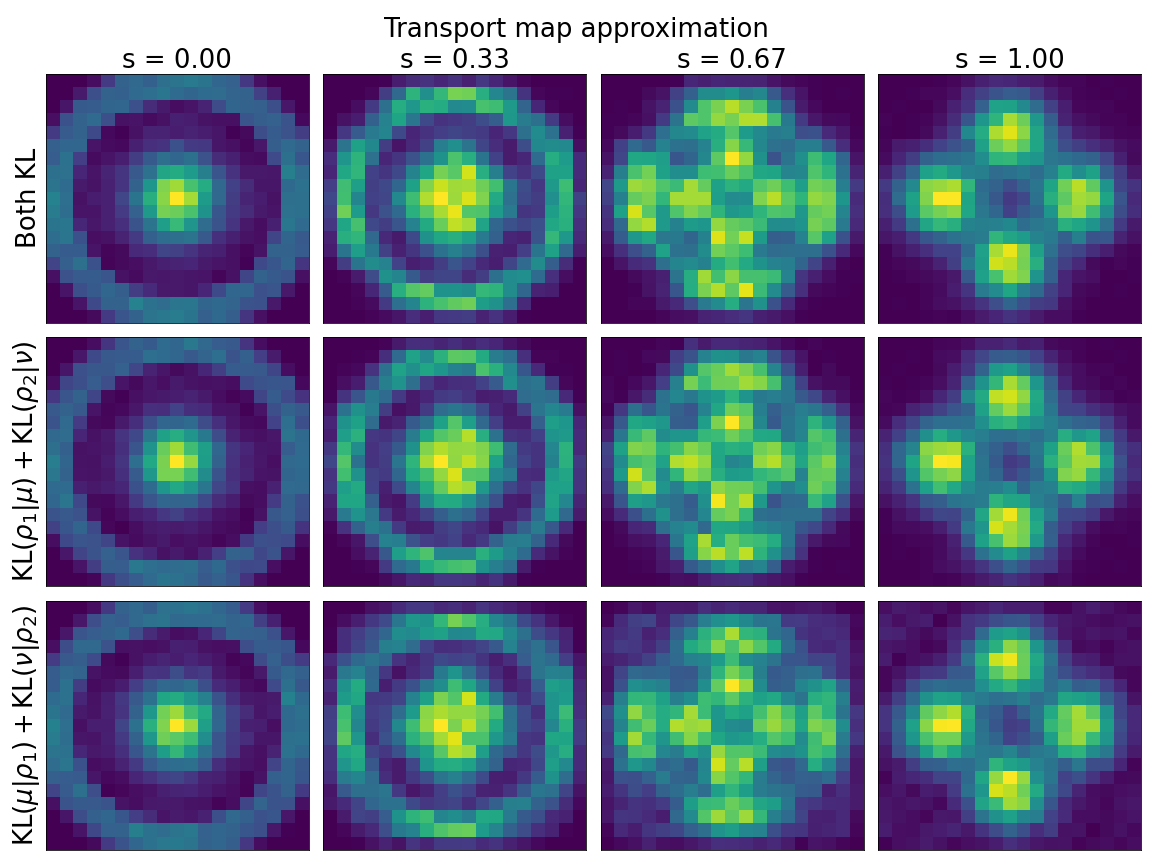}
    \caption{We test our method on two input measures which are non-log-concave (Left side is $\mu$ and right side is $\nu$). When running our method with two different entropy regularizers, using $\KL{\rho}{\mu,\nu}$ for the $Y^{(2)}$ particles and $\KL{\mu,\nu}{\rho}$ for the $X^{(2)}$ particles (labeled "both KL") is compared with using the same for all particles.  While the displacement convexity of $\KL{\cdot}{\mu}$ is well known, the displacement convexity of $\KL{\mu}{\cdot}$ is not. }
    \label{fig:KL_switching}
\end{figure}

\begin{table}[]
    \centering
    \begin{tabular}{m{5em}||m{5em}|m{5em}|m{5em}|m{5em}}
       & $L^2$ error & KL & reverse KL & total KL  \\
        \hline
     (I), $E$        & 0.030 & 0.14 & 0.13 & 0.27 \\
     (II), $\tilde E$ & 0.033 & 0.88 & 0.23 & 1.11 \\
     (III), both       & 0.028 & 0.14 & 0.13 & 0.27
    \end{tabular}
    \caption{Using method (III)
    results in the smallest $L^2$ error among the three methods because it takes advantage of $\tilde E$ being a better flow in the direction from $\nu$ ($s=1$) to $\mu$ ($s=0$) and $E$ being a better flow in the direction from $\mu$ ($s=0$) to $\nu$ ($s=1$). The large error for the $\tilde E$ flow is due to the poor performance for $\tilde E$ in the $\mu$ to $\nu$ direction, and this is avoided in method (III). }
    \label{tab:KL_switching}
\end{table}

\section*{Acknowledgments}
LC is supported by NDSEG and PIMCO fellowships. FH is supported by start-up funds at the California Institute of Technology and by NSF CAREER Award 2340762. RB is supported in part by a Vannevar Bush Faculty Fellowship (award N00014-22-1-2790) held by Andrew M. Stuart. EM is supported in part from NSF award 2240110.

\printbibliography
\end{document}